\documentclass[draft,11pt,leqno]{article}
\usepackage{amsmath,amssymb,amsthm}
\usepackage{amsthm,amssymb,amsmath,amsfonts} 
\usepackage{eucal}      
\usepackage{mathrsfs}   
\usepackage{longtable}
\usepackage{multirow}
\usepackage{verbatim}
\usepackage{placeins}
\usepackage{graphicx}
\usepackage[T1]{fontenc}
\usepackage{color}

\usepackage{lineno,hyperref}

\def\Rset{\mathbb{R}}

\theoremstyle{plain}
\newtheorem{thm}{Theorem}[section]

\newtheorem{cor}[thm]{Corollary}

\theoremstyle{definition}
\newtheorem{defn}[thm]{Definition}
\newtheorem{rem}[thm]{Remark}
\newtheorem{exmp}[thm]{Example}

\def\B{\mathfrak{B}}

\title{\bf About Conformable Derivatives in Banach Spaces}

\author{Hristo Kiskinov $^{1}$, Milena Petkova $^{2}$,\\ Andrey Zahariev $^{3}$, Magdalena Veselinova$^{4}$\vspace{0.2cm}\\ 
Faculty of Mathematics and Informatics, University of Plovdiv, \\Plovdiv 4000, Bulgaria\vspace{0.2cm}\\
$^{1}$  kiskinov@uni-plovdiv.bg;
$^{2}$  milenapetkova@uni-plovdiv.bg;\\
$^{3}$  zandrey@uni-plovdiv.bg;\\
$^{4}$  m.veselinova@fmi-plovdiv.org
}
\date{}

\begin{document}
\maketitle
\begin{abstract}
	\small 
	In the paper we discuss conformable derivative behavior in arbitrary Banach spaces
	and clear the connection between two conformable derivatives of different order. 
	As a consequence we obtain the important result that an abstract function has a conformable derivative at a point 
	(which does not coincide with the lower terminal of the conformable derivative) 
	if and only if it has a first order derivative at the same point.  
	
\textbf{Keywords:}
conformable derivative, Banach space

{\bf 2010 MSC:} 26A33, 26A24   
\end{abstract}

%
%

\section{Introduction} 
\label{sec:introduction}

In 2014, Khalil, Al Horani, Yousef and Sababheh \cite{KHYS14} 
introduced a definition of a local kind derivative called from the authors conformable fractional derivative. 
The reason for its introduction was that this derivative satisfies a lot of the well-known properties of the integer order derivatives. 
In 2015, Abdeljawad \cite{Abd15} made an extensive research of the newly introduced conformable fractional calculus. 
In \cite{Mar18} Martynyuk presented a physical interpretation of the conformable derivative.
In the last years there are published more than hundred research articles using this derivative - 
see for example 
\cite{AAAJO18}--\cite{AA17}, \cite{BLNS15}--\cite{Chu15},\cite{ER15}--\cite{IN16},\cite{LT17},\cite{MWR19},\cite{MER20},\cite{PP16},\cite{TN16}, 
\cite{ZFW15}--\cite{ZYZ18}  
and the references therein.

In \cite{AU15} Anderson and Ulness made a remark, that since the derivative is local, 
the correct name must be "conformable derivative" instead of "conformable fractional derivative".
See also \cite{OM15}, \cite{OM17}, \cite{Tar18}.
Maybe this is the good reason why the authors of \cite{Mar18}--\cite{MSS19-2}  
use the name "fractional-like" instead "conformable" derivative.

In \cite{KHYS14} and \cite{Abd15} was presented a simple relation between the conformable and the integer order derivative
but only if the last one exists.
But 
in \cite{A18}--\cite{AM19},\cite{kpzarxiv19} and \cite{kpzaip19} was shown, 
that the condition the integer order derivative to exist is unnecessary
and there was obtained that a function has a conformable derivative at a point if and only if 
it has a first order derivative at the same point and that holds for all points except the lower terminal.
Some considerations what happens in the lower terminal are given in \cite{kpzarxiv19} and \cite{kpzaip19} too.
In \cite{kpzarxiv19}  is considered also
an initial value problem (IVP) for nonlinear differential equation with conformable derivative  
and is presented a scheme, how to transform  this IVP to an equivalent IVP for equation with integer (first) order derivative
even in the case when the lower terminal of the conformable derivative coincides with the initial point.
In \cite{kpzaip19} an initial problem (IP) for nonlinear differential system with conformable derivatives and variable delays
is also transformed to an equivalent initial problem for integer (first) order delayed system, 
and then to well studied system of Volterra integral equations.

All the mentioned above studies were for real valued functions. 
And soon it was clear for the most mathematicians, 
that all such problems with conformable derivatives can be easelly reduced to equivalent problems with integer order derivatives. 

But what will happen for functions with values in arbitrary Banach spaces?
Will the simple relation between the conformable derivative and the integer order derivative still hold?
There are not a lot of studies for conformable derivatives in Banach spaces. 
We can refer only to \cite{JB19}, 
where using definitions and properties given only for real valued functions somehow are obtained results for Banach spaces. 
A doubt, that the mentioned simple relation between the conformable derivative and the integer order derivative 
does not hold for functions with values in Banach spaces is presented in \cite{TNBO20}.

That is why the goal in our article is 
to study the conformable derivative behavior for functions with values in arbitrary Banach spaces
and to give clear answer of the question about the connection between the conformable derivative and the integer order derivative in this case. 
Our answer is - yes, the same simple relation holds for functions in Banach spaces too.

The paper is organized as follows: 
In Section~2 
we give some needed definitions and properties 
for the conformable integral and conformable derivative in Banach spaces. 
In Section~3 
we discuss the conformable derivative behavior in arbitrary Banach spaces
and clear the connection between two conformable derivatives of different order. 
As a consequence we obtain the important result that an abstract function has a conformable derivative at a point 
(which does not coincide with the lower terminal of the conformable derivative) 
if and only if it has a first order derivative at the same point.  
Section~4 
is devoted to our comments to the obtained results.

%
%

\section{Preliminaries} 
\label{sec:preliminaries}

For convenience and to avoid possible misunderstandings, below we  
will present 
the definitions of the conformable integral and the conformable derivative 
for functions with values in arbitrary Banach space 
as well as some needed their properties. 
For details and more properties (for real valued functions), we refer to \cite{Abd15}, \cite{kpzarxiv19}, \cite{kpzaip19}. 
The basic definitions and theorems concerning Bochner integral and Bochner spaces necessary for this exposition are given too. 
Some definitions and preliminary results for week derivative in sense of distributions (generalized functions)
needed for our last section are also presented.

Below we will use the notations  
${\Rset_+ = (0,\infty)}$ and ${\overline{\Rset}_+ = [0,\infty)}$.

Let $\B$ be an arbitrary Banach space with norm ${||.||}_\B$ and
denote by 
$L(\Rset,\B)$ the linear space of all abstract functions $f:\Rset \to \B$ 
which are Bochner strongly measurable on every compact subinterval $J \subset \Rset$
and denote by
${L_1^{loc}(\Rset,\B)}$  the linear space of all abstract functions $f \in L(\Rset,\B)$
which are Bochner integrable 
(i.e. for which the norm ${||f(t)||}_\B$ is Lebesgue integrable)
on every compact subinterval $J$. 

\begin{defn} \label{d2.1}
	A function $f : \Rset \to \B$  is continuous at $t_0 \in \Rset$  if 
	$$
	\lim\limits_{t \to t_0} || f(t)-f(t_0)||_\B =0.
	$$
	The function $f$  is continuous in some interval $J \subset \Rset$ if it is continuous at every point $t \in J$.
\end{defn}
\begin{defn} \label{d2.2}
	A function  $f : \Rset \to \B$ is differentiable at $t_0 \in \Rset$  with pointwise derivative if there exists a vector $f'(t_0) \in \B$ such that
	$$
	\lim\limits_{h \to 0} \left\|  \frac{f(t_0+h)-f(t_0)}{h} -f'(t_0) \right\|_\B =0.
	$$
	The function $f$ is continuous differentiable in some interval $J \subset \Rset$ if its pointwise derivative exists for every point $t \in J$. 
\end{defn}

For each ${t > a}$ and ${f \in L_1^{loc}([a,\infty),\B)}$ the left-sided 
conformable integral of order ${\alpha \in (0,1]}$ with lower terminal ${a \in \Rset}$ 
is defined by
\begin{equation}   \label{e2.1}
I^{\alpha}_a f(t) = \int_a^t (s - a)^{\alpha - 1}\,f(s)ds ,
\end{equation}
where the integral is understood in Bochner sense
(see \cite{Abd15}, \cite{kpzarxiv19}, \cite{kpzaip19} for the case $\B=\Rset$).

\begin{defn}  \label{d2.3}
	The left-sided conformable derivative  of order ${\alpha \in (0,1]}$ at the point 
	${t \in (a,\infty)}$ for a function 
	${f \in L_1^{loc}([a,\infty),\B)}$
	is defined by 
	\begin{equation}   \label{e2.2}
	T^{\alpha}_a f(t) = \lim_{\theta \to 0}\left(\frac{f(t+\theta(t-a)^{1-\alpha}) - f(t)}{\theta}\right)
	\end{equation}
	if the limit exists, i.e. 
	$$
	\lim_{\theta \to 0} \left\|  \left(\frac{f(t+\theta(t-a)^{1-\alpha}) - f(t)}{\theta}\right) - T^{\alpha}_a f(t) \right\|_\B =0.
	$$
\end{defn}

As in the case of the classical fractional derivatives the point ${a \in \Rset}$ appearing in \eqref{e2.2} 
will  be called lower terminal of the left-sided conformable derivative. 
Usually, if for $f$ the conformable derivative of order $\alpha$ exists, then for simplicity we say that $f$ is $\alpha$-differentiable.

It may be noted, that some authors (see for example \cite{Abd15}) 
use the notation $T_{\alpha}^a $  instead $T^{\alpha}_a $, 
but we prefer to follow the traditions from the notations of the classical fractional derivatives 
and will write the lower terminal below and the order above.

\begin{defn}  \label{d2.4}
	The ${\alpha}$-derivative of $f$ at the lower terminal point $a$ 
	in the case when $f$ is  ${\alpha}$-differentiable 	in some interval ${(a,a+\varepsilon)}$, ${\varepsilon > 0}$ 
	is defined as 
	\[
	T^{\alpha}_a f(a) = \lim_{t \to a_+} T^{\alpha}_a f(t)
	\]
	if the limit ${\displaystyle \lim_{t \to a_+} T^{\alpha}_a f(t)}$ exists.
\end{defn}

\begin{defn}    \label{d2.5}
	The left (right) left-sided conformable derivative of order $\alpha$ at the point ${t \in (a,\infty)}$ is defined by   
	\begin{equation}
	\begin{split}
	& T^{\alpha}_a f(t-0) = \lim_{\theta \to 0-}\left(\frac{f(t+\theta(t-a)^{1-\alpha}) - f(t)}{\theta}\right) \\
	& \left(T^{\alpha}_a f(t+0) = \lim_{\theta \to 0+}\left(\frac{f(t+\theta(t-a)^{1-\alpha}) - f(t)}{\theta}\right)\right).
	\end{split}
	\end{equation}
\end{defn} 
\noindent 
Obviously $f$ is left-sided $\alpha$-differentiable 
at the point ${t \in (a,\infty)}$ if and only if $f$ is left and right left-sided  $\alpha$-differentiable 
at the point ${t \in (a,\infty)}$ and ${T^{\alpha}_a f(t+0) = T^{\alpha}_a f(t-0)}$.

Let $f$ be right left-sided $\alpha$-differentiable in some interval ${(a,a+\varepsilon)}$, ${\varepsilon > 0}$ 
and the limit ${\displaystyle\lim_{t \to a+}T^{\alpha}_a f(t+0)}$ exists. 

\begin{defn}	The right left-sided conformable derivative of order $\alpha$ at the lower terminal ${a \in \Rset}$ we define with 
	${\displaystyle T^{\alpha}_a f(a+0) = \lim_{t \to a+}T^{\alpha}_a f(t+0)}$ 
	if the limit exists. 
\end{defn}

Note that in the case when $f$ is $\alpha$-differentiable in some interval ${(a,a+\varepsilon)}$, ${\varepsilon > 0}$, 
we have that ${T^{\alpha}_a f(t) = T^{\alpha}_a f(t+0)}$ and hence ${T^{\alpha}_a f(a) = T^{\alpha}_a f(a+0)}$, i.e. both definitions coincide. 

If ${a = 0}$ we write ${T^{\alpha} f(t+0) = T^{\alpha}_0 f(t) = T^{\alpha} f(t)}$ as usual. 
If $f$ is $\alpha$-differentiable in some finite or infinite interval 
${J \subset [a,\infty)}$ we will write that  ${f \in C^{\alpha}_a (J, \B)}$, 
where with the indexes $a$ and $\alpha$ are denoted the lower terminal and the order of the conformable derivative respectively. 

In our exposition below we will use only left-sided conformable derivative 
(all definitions and statements for the right-sided conformable derivatives are mirror analogical) 
and for shortness we will omit the expression "left-sided". 

It is not difficult to check that if   ${\alpha \in (0,1]}$,  $J \subset (a,\infty)$, ${f,g \in C^{\alpha}_a (J, \B)}$,  
then as in the case when $\B=\Rset$  for arbitrary $c,d \in \Rset$ and ${t \in J}$ the following relations hold: 

(i) \ ${T^{\alpha}_a (cf+dg) = c\,T^{\alpha}_a f + d\,T^{\alpha}_a g}$;

(ii) \ ${T^{\alpha}_a (1) = 0}$. \\
If in addition $\B$ is a commutative real Banach algebra, then 

(iii) \ ${T^{\alpha}_a (fg) = g\,T^{\alpha}_a f + f\,T^{\alpha}_a g}$;

(iv) \ ${T^{\alpha}_a (fg^{-1}) = \left(f\,T^{\alpha}_a g - g\,T^{\alpha}_a f\right)g^{-2}}$,
for every $g \in \B$ for which the element $g^{-1}$ exists in $\B$.

The next theorem will be used below:

\begin{thm} [\cite{j18}] \label{t2.10} 
	Let the function ${f \in L_1^{loc}([a,\infty),\B)}$.
	Then the relation
	\begin{equation} \label{e2.4}
	f(t)=\lim_{h \to 0} \int_t^{t+h} f(s) {\rm d}s
	\end{equation}
	holds pointwise almost everywhere (a.e.) in $(a,\infty)$. 
\end{thm}

%
%

\section{Conformable derivative behavior in Banach spaces} \label{sec:3}

First we will clear the relation between the  existence if $\alpha$-derivative of a function $f$  at some point $t_0 \in (a,\infty)$ 
and the continuity of $f$  at the same point.


\begin{thm} \label{t3.1} 
	Let ${f \in L_1^{loc}([a,\infty),\B)}$  and for some $\alpha \in (0,]$ there exists a point $t_0$  with $t_0>a$, 
	such that $f$ is left (or right)  $\alpha$-differentiable at $t_0 \in (a,\infty)$  with lower terminal $a$.
	
	Then the function $f$ is left (or right) continuous at $t_0$.
\end{thm}

\begin{proof}  
	The proof of the following statement is analogical to the proof in the case when $\B=\Rset$, but for completeness we will present it. 
	
	We will prove the left variant, i.e. when $h<0$.
	
	Let $t_0>a,h<0$  be arbitrary with $t_0+h>a$ . 
	From the equation $h=\theta_h(t_0 -a)^{1-\alpha}$  it follows that there exists a unique $\theta_h <0$ 
	such that $\theta_h = h(t_0-a)^{\alpha-1}$ (if $\alpha=1$ then $\theta_h=h$) and hence the following relation holds:  
	\begin{equation} \label{e3.1}
	f(t_0+h) - f(t_0) = f(t_0+\theta_h(t_0 -a)^{1-\alpha}) - f(t_0) = \theta_h \frac{f(t_0+\theta_h(t_0 -a)^{1-\alpha}) - f(t_0)}{\theta_h}.
	\end{equation}
	Since when $h \to 0-$, and hence $\lim\limits_{h \to 0-} \theta_h = 0$  too, then we have that
	\begin{equation} \label{e3.2}
	\lim\limits_{h \to 0-} \frac{f(t_0+\theta_h(t_0 -a)^{1-\alpha}) - f(t_0)}{\theta_h} = T_a^\alpha f(t_0-0).
	\end{equation}
	From \eqref{e3.1} and \eqref{e3.2} it follows that
	$$
	\lim\limits_{h \to 0-}(f(t_0+h) - f(t_0)) = 
	\lim\limits_{h \to 0-} \theta_h   \lim\limits_{h \to 0-} \frac{f(t_0+\theta_h(t_0 -a)^{1-\alpha}) - f(t_0)}{\theta_h} = 0.
	$$
	Thus the function $f$ is left-side continuous at $t_0$. 
	
	The proof of the right variant, i.e. when $h>0$ is fully analogical.
\end{proof}

\begin{cor} \label{c3.2} 
	Let for some $\alpha \in (0,1]$  the function $f \in C_a^\alpha([a,\infty),\B)$. Then $f \in C([a,\infty),\B)$.
\end{cor}

\begin{proof}
	Let $t_0 \in (a,\infty)$ be arbitrary. 
	Then applying Theorem \ref{t3.1} we conclude that the function $f$  is left and right strong continuous at  $t_0$ and 
	$\lim\limits_{h \to 0-}(f(t_0+h) - f(t_0)) = \lim\limits_{h \to 0+}(f(t_0+h) - f(t_0)) = 0$. 
\end{proof}

The next theorem is a generalization of the results proved in \cite{kpzarxiv19}, \cite{kpzaip19} for the case when $\B=\Rset$.

\begin{thm}   \label{t3.3}
	Let  ${f \in L_1^{loc}([a,\infty),\B)}$ and there exist a point ${t_0 \in (a,\infty)}$ and number ${\alpha \in (0, 1]}$ 
	such that the conformable derivative ${T^{\alpha}_a f(t_0)}$ with lower terminal point $a$ exists.
	
	Then the conformable derivative ${T^{\beta}_a f(t_0)}$ exists 
	for every ${\beta \in (0, 1]}$ with ${\beta \neq \alpha}$ and 
	\begin{equation} \label{e3.3}
	{T^{\alpha}_a f(t_0) = (t_0 - a)^{\beta-\alpha}T^{\beta}_a f(t_0)}.
	\end{equation}
\end{thm}

\begin{proof}
	Let ${\beta \in (0, 1]}$ with ${\beta \neq \alpha}$ be arbitrary. 
	Then we have 
	\begin{equation}    \label{e3.4}
	\begin{split}
	T^{\beta}_a f(t_0)& = \lim_{\theta \to 0}\frac{f(t_0 + \theta(t_0 - a)^{1-\beta+\alpha-\alpha}) - f(t_0)}{\theta} \\ 
	& = \lim_{\theta \to 0}\frac{f(t_0 + \theta(t_0 - a)^{(1-\alpha)+(\alpha-\beta)}) - f(t_0)}{\theta} \\
	& = (t_0 - a)^{\alpha-\beta}\lim_{\theta \to 0}\frac{f(t_0 + \theta(t_0 - a)^{\alpha-\beta}(t_0 - a)^{1-\alpha}) - 
		f(t_0)}{\theta(t_0 - a)^{\alpha-\beta}}.
	\end{split}
	\end{equation}
	Then for every ${\theta \in \Rset}$ and for every fixed $\alpha , \beta \in (0,1]$  
	there exists a unique ${\theta_{\alpha,\beta} \in \Rset}$, 
	such that ${\theta_{\alpha,\beta} = \theta(t_0 - a)^{\alpha - \beta}}$. 
	Obviously when ${\theta \to 0}$ then ${\theta_{\alpha,\beta} \to 0}$ too.  
	Then from \eqref{e3.4} it follows that
	\[
	T^{\beta}_a f(t_0) = (t_0 - a)^{\alpha-\beta}\lim_{\theta \to 0}\frac{f(t_0 + \theta_{\alpha,\beta}(t_0 - a)^{1-\alpha}) - f(t_0)}{\theta_{\alpha,\beta}} = (t_0 - a)^{\alpha-\beta}T^{\alpha}_a f(t_0).
	\]
\end{proof}

\begin{cor}    \label{c3.4}
	For a function  ${f \in L_1^{loc}([a,\infty),\B)}$ the conformable derivative ${T^{\alpha}_a f(t_0)}$ 
	with lower terminal $a$ at a point ${t_0 \in (a,\infty)}$ for some ${\alpha \in (0,1)}$ exists 
	if and only if 
	the function ${f(t)}$ has first derivative at the point ${t_0 \in (a,\infty)}$ and
	\begin{equation} \label{e3.5}
	{T^{\alpha}_a f(t_0) = (t_0 - a)^{1-\alpha}  f'(t_0)}.
	\end{equation}
\end{cor}

\begin{proof}
	To prove sufficiency we apply Theorem~\ref{t3.3} for $\alpha=1$ and $\beta < 1$.
	To prove necessity we apply Theorem~\ref{t3.3} for $\alpha<1$ and $\beta =1$. 
\end{proof}

\begin{rem} 
	The result obtained in Corollary \ref{c3.4} is very important, because it shows a direct connection between the
	conformable derivative and the usual integer order derivative. 
	For the particular case $\B=\Rset$ the result ist proved also 
	in \cite{A18} and \cite{A19} for conformable derivatives with lower terminal zero and in  \cite{kpzarxiv19}, \cite{kpzaip19} for arbitrary lower terminal. 
\end{rem}

The next two corollaries treat the problem of the left and right inverse operator of the conformable derivative.

\begin{cor}    \label{c3.5} 
	Let ${f \in L_1^{loc}([a,\infty),\B)}$ and the following conditions hold:
	\begin{enumerate}
		\item [$1.$] The function $f$ has at most a first kind (bounded) jump at $a$.
		\item [$2.$] For some ${\alpha \in (0,1]}$ the function ${f \in C_a^{\alpha}((a,\infty),\B)}$.
	\end{enumerate}
	Then for every ${t \in (a,\infty)}$ we have that ${I^{\alpha}_a\,T^{\alpha}_a f(t) = f(t) - f(a)}$.
\end{cor}

\begin{proof}
	Let $t \in (a,\infty)$ be arbitrary. Corollary \ref{c3.4} implies that $f$ possess first derivative at $t$ and the relation \eqref{e3.5} holds.
	Then applying the operator $I_a^\alpha$ to both sides of \eqref{e3.5} we obtain that  
	$$
	I_a^\alpha T_a^\alpha f(t) = \int_a^t (s-a)^{\alpha-1} (s-a)^{1-\alpha} f'(s) {\rm d}s = f(t)-f(a).
	$$
\end{proof}

\begin{rem}
	Note that the statement of Corollary~\ref{c3.5} is presented in \cite{Abd15} (Lemma 2.8) for the case $\B=\Rset$ 
	without the condition 1 there, which is essential to guaranty that ${I^{\alpha}_a\,T^{\alpha}_a f(t)}$ exists for ${t \in (a,\infty)}$,
	as demonstrated with example in \cite{kpzaip19}.

\end{rem}

\begin{cor}   \label{c3.7}
	Let ${f \in L_1^{loc}([a,\infty),\B)}$ be locally bounded. 
	
	Then ${T^{\alpha}_a I^{\alpha}_a f(t) = f(t)}$ holds pointwise a.e. in $(a,\infty)$.
\end{cor}

\begin{proof}
	Since ${f \in L_1^{loc}([a,\infty),\B)}$ and it is locally bounded, then 
	denoting $g(t)=I_a^\alpha f(t)$ we obtain that
	\begin{equation} \label{e3.6}
	T_a^\alpha g(t)= \lim\limits_{\theta \to 0} (\frac{g(t+\theta (t-a)^{1-\alpha})-g(t)}{\theta} ) =
	(t-a)^{1-\alpha}   \lim\limits_{\theta \to 0} (\frac{g(t+h_\theta)-g(t)}{h_\theta} ),
	\end{equation}
	where $h_\theta = \theta (t-a)^{1-\alpha}$ and hence $\lim\limits_{\theta \to 0} h_\theta = 0$.
	Then applying Theorem \ref{t2.10} we obtain from \eqref{e3.6} that
	\begin{equation} \nonumber
	\begin{split}
	T_a^\alpha g(t) &= (t-a)^{1-\alpha} \lim\limits_{\theta \to 0} \frac{1}{h_\theta} 
	\left( \int_a^{t+h_\theta} (s-a)^{\alpha-1} f(s) {\rm d}s  -  \int_a^t (s-a)^{\alpha-1} f(s) {\rm d}s \right) \\
	&= (t-a)^{1-\alpha}  \lim\limits_{\theta \to 0} \frac{1}{h_\theta} \int_t^{t+h_\theta} (s-a)^{\alpha-1} f(s) {\rm d}s 
	= (t-a)^{1-\alpha} (t-a)^{\alpha-1} f(t) = f(t)
	\end{split}
	\end{equation}
	holds pointwise a.e. in $(a,\infty)$. 
\end{proof}

\begin{rem} \label{r3.8}
	Note that if $f \in C([a,\infty),\B)$ then the statement of Corollary \ref{c3.7} holds for every $t \in (a,\infty)$. 
	Furthermore the relation  ${T^{\alpha}_a I^{\alpha}_a f(a) = f(a)}$ holds if and only if $\lim\limits_{t \to a+}f(t)$  exists.
\end{rem}

Now we will study the conformable derivative behavior at its lower terminal. 

First we present an example which demonstrates that the behaviors of the conformable derivative 
in the inner points of considered interval are essentially different. 

\begin{exmp} \label{e4.1}
	Consider the interval $[a,a+1]$  and let $f(t)=t^\alpha , \alpha \in (0,1)$ for $t \in (a,a+1]$ and $f(a)=2$. 
	Then obviously $T_a^\alpha f(t) = \alpha$ for $t \in (a,a+1]$  and $T_a^\alpha f(a) = \lim\limits_{t \to 0+} T_a^\alpha f(t) = \alpha$. 
	Hence $T_a^\alpha f(t)$ exists and is bounded on $[a, a+1]$ and continuous at $t=a$, but $f(t)$  obviously is not continuous at $t=a$ . 
\end{exmp}

\begin{rem} \label{r4.2} 
	Note that this strange situation appears even in the case $\B=\Rset$ 
	when the lower terminal $a \in \Rset$ is the left side of the considered interval. 
	This example illustrates that Definition \ref{d2.4} for conformable derivative 
	introduced in \cite{KHYS14} for real valued functions is meaningless from applications point view. 
	It is clear that it will be more reasonable, as is standard in the whole mathematical analysis, 
	the conformable derivative in the lower terminal as well as in the ends of all compact intervals 
	to be understand in the sense of our Definition \ref{d2.5} given for the left and right (left-sided) conformable derivative for the inner points. 
\end{rem}

To support this idea, below we will prove some strange consequences from Definition \ref{d2.4}.

\begin{thm}   \label{t4.3}
	Let ${f \in L_1^{loc}([a,\infty),\B)}$  and there exists a number ${\alpha \in (0, 1]}$ 
	such that the conformable derivative ${T^{\alpha}_a f(a)}$ with lower terminal point $a$ exists.
	
	Then the conformable derivative ${T^{\beta}_a f(a)}$ also exists for every ${\beta \in (0,\alpha)}$ and 
	\[
	T^{\beta}_a f(a) = 0.
	\]
\end{thm}

\begin{proof} 
	Since $T_a^\alpha f(a)$ exists, then according Definition \ref{d2.4} 
	there exists $\varepsilon >0$  such that 
	$f \in C_a^\alpha((a,a+\varepsilon),\B)$ and $T_a^\alpha f(a)= \lim\limits_{t \to a+} T_a^\alpha f(t)$. 
	Then in virtue of Theorem~\ref{t3.3} 
	for every $\beta \in (0,1)$ we have that $f \in C_a^\beta((a,a+\varepsilon),\B)$  
	and for each $t \in (a,a+\varepsilon)$  the relation $T_a^\beta f(t) = (t-a)^{\alpha - \beta} T_a^\alpha f(t)$  holds. 
	Then for $\beta < \alpha$  we have that 
	\[
	\lim_{t \to a+} T_a^\beta f(t) =   \lim_{t \to a+} (t-a)^{\alpha - \beta} \lim_{t \to a+}  T_a^\alpha f(t) = 
	\lim_{t \to a+} (t-a)^{\alpha - \beta}  \ \  T_a^\alpha f(a)   = 0.
	\]
\end{proof}

\begin{cor}    \label{c4.4}
	Let a function ${f \in L_1^{loc}([a,\infty),\B)}$ has first derivative in $(a,a+\varepsilon)$ for some $\varepsilon >0$
	and $\limsup\limits_{t \to a+} f'(t) < \infty$. 
	
	Then $f \in C_a^\alpha([a,a+\varepsilon),\B)$  for all $\alpha \in (0,1)$ and $	  T^{\alpha}_a f(a) = 0$ .
\end{cor}

\begin{proof}
	The statement follows immediately from Theorem \ref{t4.3}, applied for $\alpha = 1$.
\end{proof}

\begin{rem} \label{r4.5}
	It is well known that inspired from the fact that ${C^n(\Rset,\Rset) \subset C^k(\Rset,\Rset)}$ for integer ${n > k}$ 
	the same problem for all classical fractional derivatives is deeply studied (see \cite{SKM93}). 
	Then the question what is the relation between ${C^{\alpha}_a((a,\infty),\B)}$ and ${C^{\beta}_a((a,\infty),\B)}$, 
	when ${\alpha, \beta \in (0,1]}$ with ${\alpha \neq \beta}$ is more than reasonable. The answer is - Theorem \ref{t3.3} states that for all ${\alpha, \beta \in (0,1]}$ we have that \[C^{\alpha}_a((a,\infty),\B) = C^{\beta}_a((a,\infty),\B)
	\] for arbitrary Banach spaces.

\end{rem}

%
%

\medskip
\section{Conclusions}
\label{sec:Conclusions}

In this research
we investigate the conformable derivative behavior for functions with values in arbitrary Banach spaces
and clear the connection between two conformable derivatives of different order. 
As a consequence we obtain the important result that an abstract function has a conformable derivative at a point 
(which does not coincide with the lower terminal of the conformable derivative) 
if and only if it has a first order derivative at the same point.
Some considerations what happens in the lower terminal are given too.

%
%


\begin{thebibliography}{33}
	
	
\bibitem{A18}
A. Abdelhakim,  
Precise interpretation of the conformable fractional derivative,
arXiv:1805.02309[math.CA] (2018).
	
\bibitem{A19}
A. Abdelhakim,
The flaw in the conformable calculus: It is conformable because it is not fractional, Fractional Calculus and Applied Analysis 22(2) (2019), 242--254. 
	
\bibitem{AM19}
A. A. Abdelhakim, J. A. Tenreiro Machado, A critical analysis of the conformable derivative, Nonlinear Dynamics 95 (4) (2019), 3063--3073.
		
\bibitem{Abd15}
T. Abdeljawad, On conformable fractional calculus, J. Comput. Appl. Math. 279 (2015), 57--66. 
	
\bibitem{AAAJO18}
T. Abdeljawad, R. Agarwal, J. Alzabut, F. Jarad, A. Ozbekler, 	Lyapunov-type inequalities for mixed non-linear forced differential equations within conformable derivatives, J. Inequal. Appl. (2018), Paper No. 143, 17 pp. 
	
\bibitem{AHK15}
T. Abdeljawad, M. Al Horani, R. Khalil, Conformable fractional semigroups of operators, J. Semigroup Theory Appl. (2015), Art. ID 7, 1--9.
	
\bibitem{AAJ17}
T. Abdeljawad, J. Alzabut, F. Jarad, A generalized Lyapunov-type inequality in the frame of conformable derivatives, Adv. Differ. Equ. 321 (2017), 1--10.
	
\bibitem{ABE19}
F. M. Alharbi, D. Baleanu, A. Ebaid, Physical properties of the projectile motion using the conformable derivative, Chinese Journal of Physics 58 (2019), 18--28.
	
	\bibitem{AA17}


	M. Al-Rifae, T. Abdeljawad, Fundamental results of conformable Sturm--Liouville eigenvalue problems, Complexity (2017), Art. ID 3720471, 1--7.
	
	
	\bibitem{AU15}
	D. R. Anderson, D. J. Ulness, Newly defined conformable derivatives, Adv. Dyn. Syst. Appl. 10(2) (2015), 109--137. 
	
	\bibitem{BLNS15}
	H. Batarfi, J. Losada, J. J. Nieto, W. Shammakh, Three-point boundary value problems for conformable fractional differential equations, 
	J. Funct. Spaces (2015), vol. 2015, Article ID 706383.
	
	\bibitem{BHT16}
	N. Benkhettou, S. Hassani, D. F. M. Torres, A conformable fractional calculus on arbitrary time scales, 
	J. King Saud Univ. Sci. 28 (1) (2016), 93--98.
	
	
	\bibitem{Chu15}
	W. Chung, Fractional Newton mechanics with conformable fractional derivative, J. Comput. Appl. Math. 290 (2015), 150--158.
	
	
	\bibitem{ER15}
	M. Eslami, H. Rezazadeh, The first integral method for Wu--Zhang system with conformable time-fractional derivative, Calcolo 53(3) (2015), 475--485.

	\bibitem{GUC15}
	A. Gokdogan, E. Unal, E. Celik, Conformable fractional Bessel equation and Bessel functions, arXiv preprint arXiv:1506.07382, 2015 - arxiv.org.
	
	\bibitem{GUC16}
	A. Gokdogan, E. Unal, E. Celik, Existence and uniqueness theorems for sequential linear conformable fractional differential equations, 
	Miskolc Mathematical Notes 17(1) (2016), 267--279.
	
	\bibitem{HK14}
	M. A. Hammad, R. Khalil, Abel's formula and Wronskian for conformable fractional differential equations, 
	Int. J. Differ. Equ. Appl. 13 (2014), 177--183.
	
	\bibitem{IN16}
	O. S. Iyiola, E. R. Nwaeze, Some new results on the new conformable fractional calculus with application using D'Alambert approach, 
	Progr. Fract. Differ. Appl. 2 (2016), 1--7.
	
	\bibitem{JB19}
	A. Jaiswal, D. Bahuguna, 
	Semilinear conformable fractional differential equations in Banach spaces,
	Differ. Equ. Dyn. Syst. 27 (1-3) (2019), 313--325.
	
	
	\bibitem{j18}
	H. D. Junghenn, 
	Principles of Analysis, 
	CRS Press, Taylor \& Francis Group, Boca Raton, (2018).
	
	
	\bibitem{KHYS14}
	R. Khalil, M. Al Horani, A. Yousef, M. Sababheh, A new definition of fractional derivative, Comput. Appl. Math. 264 (2014), 65--70. 
	
	
	\bibitem{kpzarxiv19}
	H. Kiskinov, M. Petkova, A. Zahariev, 
	Remarks about the existence of conformable derivatives and some consequences, 
	(2019), http://arxiv.org/abs/1907.03486 [Math.CA].
	
	\bibitem{kpzaip19}
	H. Kiskinov, M. Petkova, A. Zahariev, 
	About the Cauchy problem for nonlinear system with conformable derivatives and variable delays, 
	AIP Conference Proceedings 2172, 050006, (2019). 
	
	
	\bibitem{LT17}
	M. J. Lazo, D. F. M. Torres, Variational calculus with conformable fractional derivatives, 
	IEEE/CAA Journal of automatica sinica 4 (2) (2017), 340--352.
	
	\bibitem{Mar18}
	A. A. Martynyuk, On the stability of the solutions of fractional-like equations of perturbed motion, 
	Dopov. Nats. Akad. Nauk Ukr. Mat. Prirodozn. Tekh. Nauki 6 (2018), 9--16. (in Russian)
	
	\bibitem{MS18}
	Electron. J. Differ. Equ. (2018), no. 62, 1--12.
	
	\bibitem{MSS19-1}
	A. A. Martynyuk, G. Stamov, I. Stamova, Integral estimates of the solutions of fractional-like equations of perturbed motion, 
	Nonlinear Analysis: Modelling and Control 24 (1) (2019), 138--149.
	
	\bibitem{MSS19-2}
	A. A. Martynyuk, G. Stamov, I. Stamova,
	Practical stability analysis with respect to manifolds and boundedness of differential equations with fractional-like derivatives, 
	Rocky Mountain J. Math. 49 (1) (2019), 211--233. 
	
	\bibitem{MWR19}
	Li Mengmeng, JinRong Wang, D. O'Regan, Existence and Ulam's stability for conformable fractional differential equations with constant coefficients, 
	Bull. Malays. Math. Sci. Soc. 42 (4) (2019), 1791--1812.
	
	\bibitem{MER20}
	V. Mohammadnezhad, M. Eslami, H. Rezazadeh,  
	Stability analysis of linear conformable fractional differential equations system with time delays, 
	Bol. Soc. Parana. Mat. 38 (3) (2019), no. 6, 159--171.
	
	\bibitem{OM15}
	M. D. Ortigueira, J. T. Machado, What is a fractional derivative?, J. Comput. Phys. 293 (2015), 4--13.
	
	\bibitem{OM17}
	M. Ortigueira, J. Machado, Which Derivative?, Fractal Fract 1(1) (2017). 
	
	
	\bibitem{PP16}
	M. Pospisil, L. S. Pospisilova, Sturm's theorems for conformable fractional differential equations, Math. Commun. 21(2) (2016), 273--281.	
	
	\bibitem{SKM93}
	S. G. Samko, A. A. Kilbas, O. I. Marichev, Fractional Integrals and Derivatives: Theory and Applications, 
	Gordon and Breach Science Publishers, Switzerland, (1993).
	
	\bibitem{Tar18}
	V. E. Tarasov, No Nonlocality. No Fractional Derivative, Communications in Nonlinear Science and Numerical Simulation 62 (2018), 157--163.
	
	\bibitem{TN16}
	J. Tariboon, S. K. Ntouyas, Oscillation of impulsive conformable fractional differential equations, Open Math. 14 (2016), 497--508.
	
	\bibitem{TNBO20}
	Nguyen Huy Tuan, Tran Bao Ngoc, Dumitru Baleanu, Donal O'Regan,
	On well-posedness of the sub-diffusion equation with conformable derivative model,
	Communications in Nonlinear Science and Numerical Simulation 89, October (2020), 105332.
	
	
	\bibitem{ZFW15}
	A. Zheng, Y. Feng, W. Wang, The Hyers-Ulam stability of the conformable fractional differential equation, 
	Math. Aeterna 5(3) (2015), 485--492.
	
	\bibitem{ZYZ18}
	H. W. Zhou, S. Yang, S. Q. Zhang, Conformable derivative approach to anomalous diffusion, Physica A: Statistical Mechanics and its Applications 491 (2018), 1001--1013.
	
	
\end{thebibliography}
\end{document}